\documentclass{amsart}
\usepackage{hyperref}
\usepackage{mathtools}
\usepackage[all]{xy}
\usepackage{listings}
\newcommand{\f}{f_{6,3}}
\renewcommand{\phi}{\varphi}
\newtheorem{theorem}{Theorem}[section]
\newtheorem{lemma}[theorem]{Lemma}
\newtheorem{prop}[theorem]{Proposition}
\newtheorem{corollary}[theorem]{Corollary}

\newtheorem{cor}[theorem]{Corollary}
\newtheorem{lem}[theorem]{Lemma}
\newtheorem{no}[theorem]{Notation}
\theoremstyle{definition}

\newtheorem{ex}[theorem]{Example}

\newtheorem{conjecture}[theorem]{Conjecture} 
\newcommand{\GL}{\mathrm{GL}}
\newcommand{\SL}{\mathrm{SL}}

\makeatletter
\def\subsection{\@startsection{subsection}{3}%
  \z@{.5\linespacing\@plus.7\linespacing}{.1\linespacing}%
  {\normalfont\itshape}}
\makeatother

\title{Algebraic boundary of matrices of nonnegative rank at most three}
\author{Rob H. Eggermont, Emil Horobe\c{t} and Kaie Kubjas}
\date{\today}

\begin{document}

\maketitle

\begin{abstract}
The Zariski closure of the boundary of the set of matrices of nonnegative rank at most $3$ is reducible.  We give a minimal generating set for the ideal of each irreducible component.  In fact, this generating set
is a Gr\"obner basis with respect to the graded reverse lexicographic order. This solves a conjecture by Robeva, Sturmfels and the last author.
\end{abstract}

\section{Introduction}

The \textit{nonnegative rank} of a matrix $M \in \mathbb{R}_{\geq 0}^{m \times n}$ is the smallest $r \in \mathbb{N}$ such that there exist matrices $A \in \mathbb{R}_{\geq 0}^{m \times r}$ and $B \in \mathbb{R}_{\geq 0}^{r \times n}$ with $M=AB$.
Matrices of nonnegative rank at most $r$ form a semialgebraic set, i.e. they are defined by Boolean combinations of polynomial equations and inequalities. We denote this semialgebraic set by $\mathcal{M}_{m \times n}^r$. If a nonnegative matrix has rank $1$ or $2$, then its nonnegative rank equals its rank. In these cases, the semialgebraic set $\mathcal{M}_{m \times n}^r$ is defined by $2 \times 2$ or $3 \times 3$-minors respectively together with the nonnegativity constraints. In the first interesting case when $r=3$, a semialgebraic description is given by Robeva, Sturmfels and the last author~\cite[Theorem 4.1]{KRS}.

This description is in the parameter variables of $A$ and $B$, where $M=AB$ is any size $3$ factorization of $M$, so it is not clear from the description what (the Zariski closure of) the boundary is. Some boundary components are defined by the ideals $\langle m_{ij} \rangle$, where $1 \leq i \leq m, 1 \leq j \leq n$. We call them the trivial boundary components. We establish the following result,
previously conjectured in ~\cite{KRS}:

\begin{theorem}[\cite{KRS}, Conjecture 6.4]\label{main_theorem}
Let $m \geq 4, n \geq 3$ and consider a nontrivial irreducible component of $\overline{\partial \mathcal{M}_{m \times n}^r}$. The prime ideal of this component is minimally generated by  $\binom{m}{4} \binom{n}{4}$ quartics, namely the $4 \times 4$-minors, and either by $\binom{m}{3}$ sextics that are indexed by subsets $\{i,j,k\}$ of  $\{1,2,\ldots ,m\}$ or $\binom{n}{3}$ sextics that are indexed by subsets $\{i,j,k\}$ of $\{1,2,\ldots ,n\}$. These form a Gr\"obner basis with respect to graded reverse lexicographic order.
\end{theorem}

One motivation for studying the nonnegative matrix rank comes from statistics. A probability matrix of nonnegative rank $r$ records joint probabilities $\textrm{Prob}(X=i,Y=j)$ of two discrete random variables $X$ and $Y$ with $m$ and $n$ states respectively that are conditionally independent given a third discrete random variable $Z$ with $r$ states. The intersection of $\mathcal{M}_{m \times n}^r$ with the probability simplex $\Delta _{mn-1}$ is called the $r$-th mixture model, see \cite[Section 4.1]{DSS} for details. Nonnegative matrix factorizations appear also in audio processing~\cite{EM}, image compression and document analysis~\cite{LS}.

Understanding the Zariski closure of the boundary is necessary for solving optimization problems on $\mathcal{M}_{m \times n}^r$ with the certificate that we have found a global maxima. One example of such an optimization problem is the maximum likelihood estimation, i.e. given data from observations one would like to find a point in the $r$-th mixture model that maximizes the value of the likelihood function. To find the global optima, one would have to use the method of Lagrange multipliers on the Zariski closure of the semialgebraic set, its boundaries and intersections of boundaries.

The outline of this paper is the following: In Section~\ref{section:definitions} we define the topological and algebraic boundary of a semialgebraic set. In Section~\ref{section:algebraic_boundary} we find a minimal generating set of a boundary component of $\mathcal{M}_{m \times n}^3$ and in Section~\ref{section:Groebner_basis} we show that it forms a Gr\"obner basis with respect to the graded reverse lexicographic term order. In Section~\ref{section:conjectures} we state some observations and conjectures regarding the algebraic boundary of  $\mathcal{M}_{m \times n}^r$ for general $r$. Appendix~\ref{section:appendix} contains the \texttt{Macaulay2} code for the computations in Section~\ref{section:Groebner_basis}.

\textbf{Acknowledgments.} We thank Jan Draisma for providing us with theoretical insight and Robert Krone for sharing his example with us.

\section{Definitions}\label{section:definitions}
Given $k$, an infinite field, we denote the space of $m\times n$ matrices over $k$ by $M_{m\times n}$. For a fixed $r$ we will denote by $M_{m\times n}^r$ the variety of $m\times n$ matrices of rank at most $r$. Moreover we denote the usual matrix multiplication map by
\[
\mu: M_{m\times r}\times M_{r\times n} \to M_{m\times n}
\]
Then the image $\mathrm{Im}(\mu)$ is exactly $M_{m\times n}^r$. Now if we restrict the domain of $\mu$ to pairs of matrices with nonnegative entries $M_{m\times r}^{+}\times M_{r\times n}^{+}$, then the image of the restriction is the semialgebraic set $\mathcal{M}_{m\times n}^r$ of matrices with nonnegative rank at most $r$, inside the variety of matrices of rank at most $r$ (since the nonnegative rank is greater or equal to the rank). We denote its (Zariski) closure by $\overline{\mathcal{M}_{m\times n}^r}$.
We sum up our working objects in the following diagram:
\[
\mu(M_{m\times r}\times M_{r\times n})=M_{m\times n}^r\supseteq \mathcal{M}_{m\times n}^r=\mu(M_{m\times r}^{+}\times M_{r\times n}^{+}).
\]
The variety $M_{m\times n}^r$ is a subset of the topological space $k^{m\cdot n}$, so the set $\mathcal{M}_{m\times n}^r$ itself has a topological boundary inside $M_{m\times n}^r$. A matrix $M\in \mathcal{M}_{m\times n}^r$ lies on the \textit{boundary} of $\mathcal{M}_{m\times n}^r$ inside $M_{m\times n}^r$, if for any open ball $U\subseteq M_{m\times n}^r$ with $M\in U$, we have that \[U\cap \mathcal{M}_{m\times n}^r\neq U\cap M_{m\times n}^r.\] We will denote this topological boundary by $\partial(\mathcal{M}_{m\times n}^r)$. The topological boundary has a (Zariski) closure inside the variety $M_{m\times n}^r$. This closure is called the \textit{algebraic boundary} of $\mathcal{M}_{m\times n }^r$, and we denote it by $\overline{\partial(\mathcal{M}_{m\times n }^r)}$.

\section{Generators of the Ideal of an Algebraic Boundary Component}\label{section:algebraic_boundary}
Before the work of Robeva, Sturmfels and the last author \cite{KRS}, very little was known about the boundary of matrices of a given nonnegative rank. They study the algebraic boundary of $\mathcal{M}_{m\times n}^3$  for the first time and give an explicit description of the boundary. Before stating their result let us fix $r=3$ for the rest of this section, and denote the coordinates on $M_{m\times n}$ by $x_{ij}$, the coordinates on $M_{m\times 3}$ by $a_{ik}$, and the coordinates on $M_{3\times n}$ by $b_{kj}$, with $i \in \{1,\ldots,m\}$, $j \in \{1,\ldots,n\}$ and $k \in \{1,2,3\}$.

So we have that $M_{m\times n}^3$ is the image of the map $\mu$, where
\[
\mu:\ ((a_{ik}),(b_{kj}))\mapsto (x_{ij}),
\]
with $x_{ij}=\sum_{k=1,3} a_{ik}b_{kj}$, for $i \in \{1,\ldots,m\}$, $j \in \{1,\ldots,n\}$ and $k \in \{1,2,3\}$.
\begin{theorem}[\cite{KRS}, Theorem 6.1]\label{KRS:Thm5.1}
The algebraic boundary $\overline{\mathcal{M}_{m\times n}^3}$ is a reducible variety in $k^{m\cdot n}$. All irreducible components have dimension $3m+3n-10$, and their number equals
\[
mn+\frac{m(m-1)(m-2)(m+n-6)n(n-1)(n-2)}{4}.
\]
Besides the $mn$ components, defined by $\{x_{ij}=0\}$, there are
\begin{itemize}
\item[(a)] $36 {m \choose 3}{n \choose 4}$ components parametrized by $(x_{ij})=A B$, where $A$ has three zeros in distinct rows and columns, and $B$ has four zeros in three rows and distinct columns.
\item[(b)] $36 {m \choose 4}{n \choose 3}$ components parametrized by $(x_{ij})=A B$, where $A$ has four zeros in three columns and distinct rows, and $B$ has three zeros in distinct rows and columns.
\end{itemize}
\end{theorem}
Consider the irreducible component in Theorem~\ref{KRS:Thm5.1} (b) that is exactly the closure of the image of $\mathcal{A}\times \mathcal{B}$ under the multiplication map $\mu$, where we define
\[
\mathcal{A}=\left\{\left(
           \begin{array}{ccc}
             0 & * & * \\
             0 & * & * \\
             * & 0 & * \\
             * & * & 0 \\
             * & * & * \\
             \vdots & \ddots & \vdots \\
             * & * & * \\
           \end{array}
         \right)
\in M_{m\times 3}\right\}
\]
and
\[
\mathcal{B}=\left\{\left(
           \begin{array}{cccccc}
             0 & * & * & * & \cdots & *\\
             * & 0 & * & *& \cdots & *\\
             * & * & 0 & *& \cdots & *\\
           \end{array}
         \right)
\in M_{3\times n}\right\}.
\]
Let us denote this irreducible component by $X_{m,n}:=\overline{\mu(\mathcal{A}\times\mathcal{B})}$ and its ideal by $\mathcal{I}(X_{m,n})$.
In this article we describe $\mathcal{I}(X_{m,n})$ in Theorem~\ref{theorem:generators} and Theorem~\ref{theorem:equivariant_Groebner_basis}, which together give Theorem~\ref{main_theorem}.

\subsection{A $\GL_3$-action on $\mathcal{A}\times \mathcal{B}$}
We start our investigations by dualizing $\mu$ and observing that we get the following diagram of co-multiplications

\[
\xymatrix{k[M_{m\times 3}\times M_{3\times n}] & k[M_{m\times n}\ar[l]]: \mu^*\\
\mu^* \mathcal{I}(X_{m,n}) \ar@{}[u]|-*[@]{\subseteq} & \mathcal{I}(X_{m,n}) \ar@{}[u]|-*[@]{\subseteq} \ar[l]}
\]
Here $\mu^*\mathcal{I}(X_{m,n})$ is the pullback of $\mathcal{I}(X_{m,n})$.
In what follows we aim to describe $\mathcal{I}(X_{m,n})$, using acquired knowledge about $\mu^*\mathcal{I}(X_{m,n})$.

We define the following action of $\GL_3$ on $M_{m\times3}\times M_{3\times n}$, for $g\in \GL_3$, let
\[
g\cdot(A,B)=(Ag^{-1}, gB).
\]
This action naturally induces an action on $k[M_{m\times3}\times M_{3\times n}]$, by
\[
g\cdot f(A,B)=f(g^{-1}\cdot (A,B))=f(Ag,g^{-1}B),
\] for $g\in \GL_3$ and for $f\in k[M_{m\times3}\times M_{3\times n}]$.\\

Observe that $\mu$ and $\mu^*$ are invariant maps with respect to the action defined above, since
\begin{equation}\label{rem:muGL3}
\mu(g\cdot (A,B)) = (Ag^{-1})(gB) = AB = \mu(A,B),
\end{equation}
for all $(A,B) \in M_{m\times3}\times M_{3\times n}$ and all $g \in \GL_3$.

Once we have the above defined action, it is natural to investigate the orbit of our defining set, $\mathcal{A}\times \mathcal{B}$, under this action. For this we can formulate the following proposition.
\begin{prop}\label{prop:GLABishypersurf}
The closure of the orbit of the $\GL_3$-action on the set $\mathcal{A}\times \mathcal{B}$ is a hypersurface.
\end{prop}
\begin{proof}
It suffices to show that $\GL_3 \cdot (\mathcal{A}\times \mathcal{B})$ has codimension $1$ in $M_{m\times3} \times M_{3\times n}$. Note that $\mathcal{A}\times \mathcal{B}$ has codimension $7$, and $\GL_3$ has dimension $9$.

Observe that if $g \in \GL_3$ is diagonal, it maps $\mathcal{A}\times \mathcal{B}$ to itself. On the other hand, we can verify that if $(A,B) \in \mathcal{A}\times \mathcal{B}$ is sufficiently generic, and $g \in \GL_3$ is not diagonal, then $g\cdot (A,B)$ does \emph{not} lie in $\mathcal{A}\times \mathcal{B}$. Since the diagonal matrices form a $3$-dimensional subvariety of $\GL_3$, we find that the codimension of $\GL_3 \cdot (\mathcal{A}\times \mathcal{B})$ is $7-9+3 = 1$, as was to be shown.
\end{proof}

A hypersurface is the zero set of a single polynomial. We now give an explicit construction of an irreducible polynomial that vanishes on $\GL_3 \cdot (\mathcal{A}\times \mathcal{B})$.

 First we take
 $f = (-x_{13} x_{21} + x_{11} x_{23}) (x_{13} x_{22} - x_{12} x_{23}) x_{32} x_{41} -(-x_{13} x_{21} + x_{11} x_{23}) ((x_{13} x_{22} - x_{12} x_{23}) x_{31} - (-x_{12} x_{21} + x_{11} x_{22}) x_{33}) x_{42} + (-x_{12} x_{21} +     x_{11} x_{22}) ((x_{13} x_{22} - x_{12} x_{23}) x_{31} - (-x_{12} x_{21} + x_{11} x_{22}) x_{33}) x_{43}.$\\

 Now the pull-back $\mu^*f$ factors as
 \[(b_{13} b_{22} b_{31} - b_{12} b_{23} b_{31} - b_{13} b_{21} b_{32} + b_{11} b_{23} b_{32} + b_{12} b_{21} b_{33} - b_{11} b_{22} b_{33})\f,\]
 with $\f$ a homogeneous degree $(6,3)$-polynomial in the variables $a_{i,k}$ and $b_{k,j}$ with $i \in \{1,\ldots,m\}, j \in \{1,\ldots,n\},$ and $k \in \{1,2,3\}$.

Observe that $\f$ vanishes on $\mathcal{A} \times \mathcal{B}$ and the following lemma will imply that it vanishes on $\GL_3\cdot (\mathcal{A} \times \mathcal{B})$.

\begin{lemma}\label{lem:GL3onf63} The polynomial $\f$ is $\SL_3$-invariant. Moreover, for any $g \in \GL_3$, we have $g\cdot \f = \det(g)\f$.
\end{lemma}

\begin{proof} Note that $D = (b_{13} b_{22} b_{31} - b_{12} b_{23} b_{31} - b_{13} b_{21} b_{32} + b_{11} b_{23} b_{32} + b_{12} b_{21} b_{33} - b_{11} b_{22} b_{33})$ is a $3 \times 3$-determinant, and hence it is $\SL_3$-invariant. Moreover, we have \[(g\cdot D)(B) = D (g^{-1}B) = \det(g)^{-1}D(B),\] for any $g \in \GL_3$ and $B\in \mathcal{B}$.

Moreover, $\mu^* f$ is non-zero and $\GL_3$-invariant, by \ref{rem:muGL3}. So we have \[D  \f = \mu^*f = g \cdot\mu^* f = (g\cdot D) (g\cdot \f) = (\det(g)^{-1} D) (g\cdot \f),\] for any $g \in \GL_3$. It follows that we must have \[g\cdot \f = \det(g)\f.\] In particular $\f$ is $\SL_3$-invariant.
\end{proof}

Since $\f$ vanishes on $\mathcal{A}\times \mathcal{B}$, we immediately have the following corollary.

\begin{corollary}\label{cor:f63vanishesonGL3AB} The ideal of the set $\GL_3 \cdot (\mathcal{A} \times \mathcal{B})$ is $(\f)$.
\end{corollary}

\begin{proof}By Proposition~\ref{prop:GLABishypersurf}, the set $\overline{\GL_3 \cdot (\mathcal{A}\times\mathcal{B})}$ is a hypersurface. By the previous lemma, the polynomial $\f$ vanishes on $\GL_3 \cdot (\mathcal{A}\times\mathcal{B})$, since for any $(A,B) \in \mathcal{A}\times\mathcal{B}$ and any $g \in \GL_3$, we have $\f(Ag^{-1},gB) = (g^{-1}\cdot\f)(A,B) = \det(g)^{-1}\f(A,B) = \det(g)^{-1} \cdot 0 = 0$. One can easily check that $\f$ is irreducible, so the set $\overline{\GL_3 \cdot (\mathcal{A}\times\mathcal{B})}$ must be the zero set of $\f$, and hence its ideal, which is the ideal of $\GL_3 \cdot (\mathcal{A}\times\mathcal{B})$ as well, must be $(\f)$.
\end{proof}

\subsection{The ideal of $X_{m,n}$}

In what follows we will relate the ideal of $\GL_3 \cdot (\mathcal{A}\times\mathcal{B})$ with the pull-back of the ideal of $X_{m,n}$. To do this we formulate two technical lemmas. The first one contains the algebraic geometric essence of the proofs which follow. The other one extracts the representation theory between the lines.

\begin{lem}\label{lem:pullbackideal} Let $S$   be a subset of $M_{m\times 3}\times M_{3\times n}$, let $Y$ be a subset of $ M_{m\times n}$, and suppose $\mu(S)$ is a Zariski dense subset of $Y$. Then $\mathcal{I}(Y) = (\mu^*)^{-1}(\mathcal{I}(S))$ and $\mu^*\mathcal{I}(Y) = \mathcal{I}(S) \cap \mathrm{Im}(\mu^*)$.
\end{lem}

\begin{proof} Since $\mu(S)$ is dense in $Y$, applying $\mu^*$ we have
\[
\mu^*(\mathcal{I}(\mu(S)))=\mu^*(\mathcal{I}(Y)).
\]
It remains to prove that $\mu^*(\mathcal{I}(\mu(S)))=\mathcal{I}(S) \cap \mathrm{Im}(\mu^*)$. For this take $f\in \mathcal{I}(\mu(S))$, so for any $(A,B)\in S$ we have that $\mu^*f(A,B)=f(\mu(A,B))=0$, hence \[\mu^*(\mathcal{I}(\mu(S)))\subseteq \mathcal{I}(S) \cap \mathrm{Im}(\mu^*).\] Conversely take $f=\mu^* f'$ in $\mathcal{I}(S) \cap \mathrm{Im}(\mu^*)$, so for any $(A,B)\in S$ we have that $0=f(A,B)=(\mu^*f')(A,B)=f'(\mu(A,B)),$ hence
\[\mu^*(\mathcal{I}(\mu(S)))\supseteq \mathcal{I}(S) \cap \mathrm{Im}(\mu^*).\]
So we find that $\mu^*(\mathcal{I}(Y))=\mathcal{I}(S) \cap \mathrm{Im}(\mu^*)$.
Clearly, this means $\mathcal{I}(Y) = (\mu^*)^{-1}(\mathcal{I}(S))$ as well.
\end{proof}

\begin{lem}\label{lem:GL3invariance} The image of $\mu^*$ is equal to $k[M_{m\times3}\times M_{3\times n}]^{\GL_3}$.
\end{lem}

\begin{proof} First, observe that for any $f \in k[M_{m\times n}]$, any $(A,B) \in M_{m\times3}\times M_{3\times n}$, and any $g \in \GL_3$, we have \[g\cdot(\mu^*f)(A,B) = f(\mu(g\cdot(A,B))) = f(\mu(A,B)) = \mu^*f(A,B),\] and hence $\mathrm{Im}(\mu^*) \subseteq k[M_{m\times3}\times M_{3\times n}]^{\GL_3}$.

To prove the other inclusion, we refer to the First Fundamental theorem for $\GL_3$ (see for instance [\cite{KP}, Section 2.1 or \cite{DG}, Section 11.2.1]), which states that the $\GL_3$-invariant polynomials of $k[M_{m,3}\times M_{3,n}]$ are generated by the inner products
\[
\sum_{k=1}^3 a_{i,k}b_{k,j},
\] for all $1\leq i\leq m$ and $1\leq j\leq n$. Since these are simply the $\mu^*(x_{i,j})$, we find that $\mathrm{Im}(\mu^*) \supseteq k[M_{m\times3}\times M_{3\times n}]^{\GL_3}$, which completes the proof.
\end{proof}
Now as promised the following lemma relates $\mu^*\mathcal{I}(X_{m,n})$ with $\GL_3 \cdot (\mathcal{A}\times\mathcal{B})$. We have the following equality.

\begin{lem}~\label{lem:main}
The pull-back of the ideal $\mathcal{I}(X_{m,n})$ is exactly $(\f)^{GL_3}$.
\end{lem}
\begin{proof}
We have $\mu(\GL_3 \cdot (\mathcal{A} \times \mathcal{B})) = \mu(\mathcal{A}\times \mathcal{B})$ is dense in $X$. By Lemma~\ref{lem:pullbackideal} we get
\[
\mu^*\mathcal{I}(X_{m,n}) = \mathcal{I}(\GL_3 \cdot (\mathcal{A} \times \mathcal{B})) \cap \mathrm{Im}(\mu^*).
\] Then applying Corollary~\ref{cor:f63vanishesonGL3AB} for the structure of $\GL_3 \cdot (\mathcal{A} \times \mathcal{B})$ and Lemma~\ref{lem:GL3invariance} for pull-back of $\mu$, we get that
\[
\mu^*\mathcal{I}(X_{m,n}) = (\f) \cap k[M_{m\times3}\times M_{3\times n}]^{\GL_3},
\] which finishes the proof.
\end{proof}
We remark that a consequence of the above ideas is the primality of $\mathcal{I}(X_{m,n})$.
\begin{cor}~\label{cor:Iisprime} The ideal $\mathcal{I}(X_{m,n})$ is prime.
\end{cor}
\begin{proof} Lemma~\ref{lem:main} together with Lemma~\ref{lem:pullbackideal} implies that $\mathcal{I}(X_{m,n}) = (\mu^*)^{-1}((\f))$. But then $(\f)$ is prime, since $\f$ is irreducible. This implies that $\mathcal{I}(X_{m,n}) = (\mu^*)^{-1}((\f))$ is prime as well.
\end{proof}

We continue investigating the structure of $(\f)^{\GL_3}$. For this we introduce the following notation.
For $\mathbf{i} = (i,j,k)$ an ordered triple of elements in $\{1,\ldots,n\}$, we denote $\det_{B,\mathbf{i}} = \det \left(\begin{smallmatrix}b_{1i} & b_{1j} & b_{1k} \\ b_{2i} & b_{2j} & b_{2k} \\ b_{3i} & b_{3j} & b_{3k}\end{smallmatrix}\right)$. Analogously, for $\mathbf{i} = (i,j,k)$ an ordered triple of elements in $\{1,\ldots,m\}$, we denote $\det_{A,\mathbf{i}} = \det \left(\begin{smallmatrix}a_{i1} & a_{i2} & a_{i3} \\ a_{j1} & a_{j2} & a_{j3} \\ a_{k1} & a_{k2} & a_{k3}\end{smallmatrix}\right)$.

The following proposition is the main result of this part, describing explicitly the pull-back of $\mathcal{I}(X_{m,n})$.
\begin{prop}\label{prop:fGL3}
We have $\mu^*\mathcal{I}(X_{m,n}) = \left\{\sum_{\mathbf{i}}\f \det_{B,\mathbf{i}}h_{\mathbf{i}}: h_{\mathbf{i}} \in k[M_{m\times3}\times M_{3\times n}]^{\GL_3}\right\}$. Moreover, the $\f\det_{B,\mathbf{i}}$ are $\GL_3$-invariant. Here, $\mathbf{i}$ runs over the ordered triples of elements in $\{1,\ldots,n\}$.
\end{prop}

\begin{proof} First by Lemma \ref{lem:main} we have that $\mu^*\mathcal{I}(X_{m,n})=(\f)^{GL_3}$, then we recall that, by Lemma~\ref{lem:GL3onf63},  $\f$ is $\SL_3$-invariant and that for any $g \in \GL_3$ we have $g\cdot\f = \det(g)\f$.
Therefore, any $\GL_3$-invariant element $f$ of $(\f)$, has the form
\[
f =\f h,
\] with $h$ an $\SL_3$-invariant polynomial satisfying $g\cdot h = \det(g)^{-1}h$ for any $g \in \GL_3$.

By the First Fundamental Theorem for $\SL_n$ (see for instance [\cite{KP}, Section 8.4]) we know that $h$ can be expressed in terms of the $\det_{A,\mathbf{i}}$, the $\det_{B,\mathbf{i}}$ and the scalar products $\sum_{k} a_{i,k}b_{k,j}$. Observe that $\GL_3$ acts trivially on the $\sum_{k} a_{i,k}b_{k,j}$, and acts on the $\det_{A,\mathbf{i}}$ and $\det_{B,\mathbf{i}}$ by $g\cdot \det_{A,\mathbf{i}} = \det(g) \det_{A,\mathbf{i}}$ and $g\cdot \det_{B,\mathbf{i}} = \det(g)^{-1} \det_{B,\mathbf{i}}$. The polynomial ring generated by these elements is therefore $\mathbb{Z}$-graded, where the part of degree $d$ is the part of the ring on which any $g$ acts by multiplication with $\det(g)^d$.

Since $g\cdot h = \det(g)^{-1}h$, it follows that $h$ has degree $-1$, and hence we can express it in the form \[\sum_i\mathrm{det}_{B,\mathbf{i}}\cdot h_{\mathbf{i}},\] where the $h_{\mathbf{i}}$ are of degree $0$, and hence are $\GL_3$-invariant polynomials.

Then our $f$ has the form
\begin{equation*}
f=\sum_{\mathbf{i}}(\f\mathrm{det}_{B,\mathbf{i}})\cdot h_{\mathbf{i}}\text{, with } h_{\mathbf{i}}\in k[M_{m,3}\times M_{3,n}]^{\GL_3}.
\end{equation*}

So any $f \in (\f)^{\GL_3}$ can be expressed in the desired form, and each element of this form is $\GL_3$-invariant. Moreover, the $\f\det_{B,\mathbf{i}}$ are $\GL_3$-invariant, as was to be shown.
\end{proof}

Finally we have arrived at the point to draw conclusions about the generators of $\mathcal{I}(X_{m,n})$, using the knowledge we acquired about $\mu^*\mathcal{I}(X_{m,n})$.
Take an arbitrary element $f$ of $\mathcal{I}(X_{m,n})$. By Proposition~\ref{prop:fGL3}, the polynomial $\mu^*f$ can be written as
\[
\mu^* f=\sum_i \left(\f \mathrm{det}_{B,\mathbf{i}}\right) h_{\mathbf{i}},
\] for some $h_{\mathrm{i}}\in k[M_{m,3}\times M_{3,n}]^{\GL_3}$.
For each $\mathbf{i}$, fix $f_{\mathbf{i}}$ such that $\mu^*f_{\mathbf{i}} = \f \mathrm{det}_{B,\mathbf{i}}$.

Since $k[M_{m,3}\times M_{3,n}]^{\GL_3}$ is the image of $\mu^*$ (by \ref{lem:GL3invariance}), there exist $\alpha_{\mathbf{i}}$, such that $\mu^* \alpha_{\mathbf{i}}= h_{\mathbf{i}}$ for each $\mathbf{i}$. This way finally we get that
\[
\mu^*f=\mu^*\left(\sum_{\mathbf{i}} f_{\mathbf{i}} \alpha_{\mathbf{i}} \right).
\]
And finally this reads as
\[
f-\sum_{\mathbf{i}} f_{\mathbf{i}} \alpha_{\mathbf{i}} \in \mathrm{Ker}(\mu^*).
\]
The kernel of $\mu^*$ is generated by all the $4\times 4$ determinants $\mathrm{det}_{\mathbf{j},\mathbf{k}}$ of matrices in $M_{m,n}$ (where $\mathbf{j}$, respectively $\mathbf{k}$, is an ordered $4$-tuple of elements in $\{1,\ldots,m\}$, respectively in $\{1,\ldots,n\}$, and the determinant is defined as one would expect), so we conclude that
\begin{equation}
\mathcal{I}(X_{m,n})\subseteq\left(f_{\mathbf{i}}, \mathrm{det}_{\mathbf{j},\mathbf{k}}\right)_{\mathbf{i},\mathbf{j},\mathbf{k}}.
\end{equation}\\
The other inclusion is obvious from the fact that the $f_{\mathbf{i}}$ and $\mathrm{det}_{\mathbf{j},\mathbf{k}}$ vanish on $X$. This means we have just proved the following theorem.
\begin{theorem}\label{theorem:generators}
The ideal of the variety $X_{m,n}$ is generated by degree $6$ and degree $4$ polynomials, namely
\[
\mathcal{I}(X_{m,n})=\left(f_{\mathbf{i}}, \mathrm{det}_{\mathbf{j},\mathbf{k}}\right)_{\mathbf{i},\mathbf{j},\mathbf{k}}.
\]
\end{theorem}

\section{Gr\"obner Basis}\label{section:Groebner_basis}

In this section we will show that the generators in Theorem~\ref{theorem:generators} form a Gr\"obner basis with respect to the graded reverse lexicographic term order. Let $G$ be the monoid of all maps $\pi$ from $\mathbb{N} \times \mathbb{N}$ to itself such that
\begin{itemize}
 \item $\pi(ij)=ij'$ for $i \in [4]$,

 \item $\pi(ij)=i'j$ for $j \in [3]$,

 \item $\pi$ is coordinatewise strictly increasing.
\end{itemize}
Here we slightly abuse the notation and write $ij$ for a pair $(i,j)$. Let us denote
\[k[x]=k[x_{ij}:i,j \in \mathbb{N}].\]
Then $G$ acts on $k[x]$ by $\pi\cdot x_{ij} = x_{\pi (ij)}$. By Theorem~\ref{theorem:generators}, we have \[\mathcal{I}(X_{m,n})=G\cdot \mathcal{I}(X_{4,6}) \cap k[x_{ij}: 1\leq i \leq m, 1 \leq j \leq n].\]

\begin{theorem}\label{theorem:equivariant_Groebner_basis}
The $4 \times 4$-minors and sextics indexed by $\{i,j,k\} \subset \mathbb{N}$ form an equivariant Gr\"obner basis of $G\cdot \mathcal{I}(X_{4,6})$ with respect to the graded reverse lexicographic term order. The $4 \times 4$-minors and sextics indexed by $\{i,j,k\} \subset \{1,\ldots ,n\}$ form a Gr\"obner basis of $\mathcal{I}(X_{m,n})$ with respect to the graded reverse lexicographic term order.
\end{theorem}

\begin{lemma}\label{lemma:EGB4}
For all $b_0,b_1 \in k[x_{ij}: i,j \in \mathbb{N}]$ the set $G\cdot b_0 \times G\cdot b_1$ is the union of a finite number of $G$-orbits.
\end{lemma}

\begin{proof}
Consider all quadruples $(S_0,S_1,T_0,T_1)$ of sets $S_0,S_1,T_0,T_1 \subseteq \mathbb{N}$ with $|S_i|$ equal to the number of different first coordinates of indices of $b_i$ and $|T_i|$ equal to the number of different second coordinates of indices of $b_i$ for which $S_0 \cup S_1$ and $T_0 \cup T_1$ are intervals of the form $\{1,\ldots ,k\}$ for some $k$ (in general, $k$ is different for $S_0 \cup S_1$ and $T_0 \cup T_1$). Note that there are only finitely many such quadruples $(S_0,S_1,T_0,T_1)$. For each such quadruple, let $(\pi_0,\pi_1)$ be a pair of elements of $G$ such that projecting all indices appearing in $\pi_i(b_i)$ to the first coordinate gives $S_i$ and to the second coordinate $T_i$ (if such pair exists). Then we have
\begin{displaymath}
G\cdot b_0 \times G\cdot b_1 = \bigcup _{(S_0,S_1,T_0,T_1)} G\cdot(\pi_0 b_0,\pi_1 b_1),
\end{displaymath}
where the union is over all quadruples $(S_0,S_1,T_0,T_1)$ as above.
\end{proof}

Proof of Theorem~\ref{theorem:equivariant_Groebner_basis} follows step-by-step the proof of~\cite[Theorem~3.1]{BD11}.

\begin{proof}[Proof of Theorem~\ref{theorem:equivariant_Groebner_basis}]
Let $B$ denote the generators of the ideal $\mathcal{I}(X_{4,6})$. By the equivariant Buchberger criterion~\cite[Theorem~2.5]{BD11}, we have to show that for all $b_0,b_1 \in B$ there exists a complete set of $S$-polynomials each of which has $0$ as a $G$-reminder modulo $B$. By the proof of Lemma~\ref{lemma:EGB4}, we need only $G$-reduce modulo $B$ all $S$-polynomials of  pairs $(\pi_0 \cdot b_0,\pi_1 \cdot b_1)$ with $b_0,b_1 \in B$ and $\pi_0,\pi_1 \in G$ such that $\pi_0 \cdot b_0 \cup \pi_1 \cdot b_1$ projected to each coordinate forms an interval of the form $\{1,\ldots ,k\}$. If $b_0, b_1$ are both $4 \times 4$-minors, then $\pi_0\cdot b_0, \pi_1\cdot b_1$ are also $4 \times 4$-minors and their $S$-polynomial has $G$-remainder $0$ modulo $B$. If one of $b_0$ and $b_1$ is a sextic and the other one is a $4 \times 4$-minor, then the maximal element in $S_0 \cup S_1$ is less then or equal to $8$ and the maximal element in $T_0 \cup T_1$ is less then or equal to $10$. If $b_0, b_1$ are both sextics, then $S_0 \cup S_1=\{1,2,3,4\}$ and the maximal element in $T_0 \cup T_1$ is less then or equal to $9$, because every sextic is homogeneous of degree $e_1{+}e_2{+}e_3{+}e_i{+}e_j{+}e_k$ in the column
grading.

Finally, we use \texttt{Macaulay2} to show that the Gr\"obner basis of $\mathcal{I}(X_{8,10})$ is generated by the $4 \times 4$-minors and sextics indexed by $\{i,j,k\} \subset \{1,\ldots ,8\}$, for details see Appendix~\ref{section:appendix}.
\end{proof}

\section{Open Problems}\label{section:conjectures}
In the previous sections we have seen structure theorems for the algebraic boundary of matrices of nonnegative rank three. In
this section we investigate the algebraic boundary of matrices of arbitrary nonnegative rank $r$. Hoping for similar results as for the rank three case is an ambitious project, therefore we aim for studying the stabilization behavior of the nonnegative rank boundary.
For matrix rank is true that if the dimensions, $m$ and $n$, of a matrix $M$ are sufficiently large, then already a submatrix of $M$ has the same rank as $M$ does. We want to prove something similar for the nonnegative rank.

When letting both $m$ and $n$ tend to infinity, it is not true that the nonnegative rank of a given matrix can be tested by calculating the nonnegative rank of its submatrices. More precisely given a nonnegative matrix $M=(m_i)_{1\leq i\leq n}$ on the topological boundary $\partial (\mathcal{M}_{m\times n}^r)$ it might happen that all of the submatrices $M^{\widehat{i_0}}=(m_i)_{i=1,...,n}^{i\neq i_0}$ have smaller nonnegative rank then $M$ has (or the same for rows).

We will give a family of examples showing this. In~\cite{MOI}, Ankur Moitra gives a family of examples of $3n\times 3n$ matrices of nonnegative rank $4$ for which every $3n\times n$ submatrix has nonnegative rank $3$. We will strengthen his result to be true for every $3n \times (\lceil\frac{3}{2}n\rceil-1)$ submatrix.

To present this example we remind our readers about  the geometric approach to nonnegative rank.
Finding the nonnegative rank of a matrix  is equivalent to finding a polytope with minimal number of vertices nested between  two given polytopes. For this approach to nonnegative rank see for instance \cite[Section 2]{MSS}.
Let $M\in M_{m\times n}^r$ be a rank $r$ nonnegative matrix and let $\Delta_{m-1}=\mathbb{R}_+^m\cap H$, where
\[H=\left\{x\in\mathbb{R}^m| \sum_{i=1}^m x_i=1\right\}.\]
Then define
\[
W=\mathrm{Span}(M)\cap \Delta_{m-1}\text{  and  } V=\mathrm{Cone}(M)\cap \Delta_{m-1},
\]
where $\mathrm{Span}(M)$ and $\mathrm{Cone}(M)$ are the linear space and positive cone spanned by the column vectors of $M$. We have the following lemma.

\begin{lemma}[\cite{MSS}, Lemma 2.2]\label{lm:nonnegativerankandsimplices}
Let $\textrm{rank}(M)=r$. The matrix $M$ has nonnegative rank exactly $r$ if and only if there exists a $(r-1)$-simplex $\Delta$ such that $V\subseteq \Delta\subseteq W$.
\end{lemma}

In the case of nonnegative rank $3$, Robeva, Sturmfels and the last author study the boundary of the mixture model, based on~\cite[ Lemma 3.10 and Lemma 4.3]{MSS}.

\begin{prop}[\cite{KRS}, Corollary 4.4]\label{prop:boundarycases} Let $M \in \mathcal{M}_{m\times n}^3$. Then $M \in \partial\mathcal{M}_{m\times n}^3$ if and only if
\begin{itemize}
\item $M$ has a zero entry, or
\item $\mathrm{rank}(M) = 3$ and if $\Delta$ is any triangle with $V \subseteq \Delta \subseteq W$, then every edge of $\Delta$ contains a vertex of $V$, and either an edge of $\Delta$ contains an edge of $V$, or a vertex of $\Delta$ coincides with a vertex of $W$.
\end{itemize}
\end{prop}

Note that all vertices of $\Delta$ in the above proposition must lie on $W$. Together with results from \cite{MSS}, one can show that if $M$ has rank $3$ and it lies in the interior of $\mathcal{M}_{m\times n}^3$ then there is a $\Delta$ with $V \subseteq \Delta \subseteq W$ such that every vertex of $\Delta$ lies on $W$, and either an edge of $\Delta$ contains an edge of $V$ or a vertex of $\Delta$ coincides with a vertex of $W$. These are the types of triangles we will be interested in.

\begin{no} Let $V \subseteq W$ be convex polygons such that $V$ is contained in the interior of $W$.
\begin{description}
\item[1] For a vertex $w$ of $W$, let $l_1,l_2$ be the rays of the minimal cone centered at $w$ and containing $V$. Let $w_i$ be the point on $l_i \cap W$ furthest away from $w$. We denote the triangle formed by $w,w_1$ and $w_2$ by $\Delta^w_{V,W}$.
\item[2] For an edge $e = (v_1,v_2)$ of $V$, consider the line $l$ containing $e$. Let $w_1,w_2$ be the points where $l$ intersects $W$. The minimal cone centered at $w_i$ containing $V$ has two rays, one of which contains $e$. Let $l_i$ be the one not containing $e$. If $l_1$ and $l_2$ intersect inside $W$, we denote the triangle formed by $w_1,w_2$ and $l_1\cap l_2$ by $\Delta^{e}_{V,W}$.
\end{description}
We omit subscripts when possible.
\end{no}

As a consequence of the discussion above, to test whether or not the pair $(W,V)$ corresponds to a matrix and its nonnegative rank $3$ factorization, it suffices to look at the triangles $\Delta^w, \Delta^e$ with $w$ running over the vertices of $W$ and $e$ running over the edges of $V$.

We are now ready to show Moitra's family of examples. For simplicity, we work with regular $3n$-gons, which is slightly more restrictive than Moitra's actual family. Regardless of this, the conclusions will hold even if we consider the full family.

\begin{ex}
Let $W$ be a regular $3n$-gon for some $n > 1$. Label the vertices $w_1,\ldots,w_{3n}$ in clockwise order. Let $V$ be the polygon cut out by the lines $l_i = w_iw_{i+n}$ for $i \in \{1,\ldots,3n\}$ (computing modulo $3n$). Note that each $l_i$ contains some edge of $V$. Since all $l_i$ are distinct, it follows that $V$ is a $3n$-gon. Observe that for any $i$, the triangle $\Delta^{w_i}$ is the triangle formed by the lines $l_i, l_{i+n}, l_{i+2n}$ (or alternatively, by the points $w_i,w_{i+n},w_{i+2n}$). Moreover, for any edge $e$ of $V$, any of the triangles $\Delta^e$ is one of the $\Delta^{w_i}$. See the left hand side of Figure~\ref{fig:moitra12} for an example.

It is now easily verified that these triangles are the only triangles $\Delta$ with $V \subseteq \Delta \subseteq W$. Indeed, Moitra showed that the pair $(W,V)$ corresponds to a matrix $M$ in $\partial\mathcal{M}_{3n\times 3n}^3$, which is equivalent to the above statement by Proposition~\ref{prop:boundarycases} (and by the fact that $V$ is contained in the interior of $W$, which implies that the corresponding matrix does not have any zero entries).

We expand $V$ to $V'$ by moving each vertex of $V$ a factor $\epsilon$ away from the center. Since any triangle containing $V'$ must also contain $V$, and since the $\Delta^{w_i}_{W,V}$ do not contain $V'$, there are no triangles $\Delta$ with $V' \subseteq \Delta \subseteq W$, and hence $(W,V')$ corresponds to a matrix $M'$ of nonnegative rank at least $4$.

We observe that if $\epsilon$ is small enough, the triangle $\Delta^{w_i}_{W,V'}$ contains all but two vertices of $V'$, namely the two vertices of $V'$ corresponding to the vertices of $V$ that lie on the line $w_{i+n}w_{i+2n}$. An example of such a triangle can be seen on the right hand side of Figure~\ref{fig:moitra12}.

\begin{figure}[h]
\includegraphics[width=12cm]{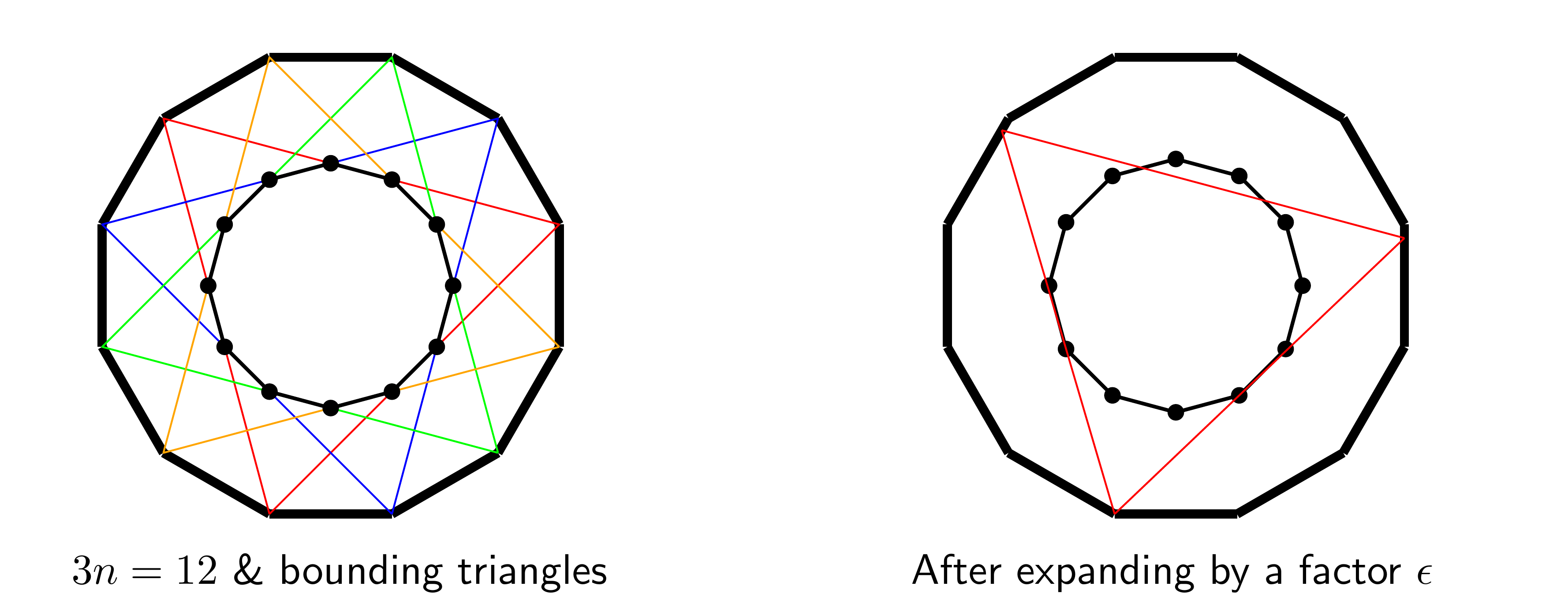}
\caption{Moitra's example with 12 vertices}\label{fig:moitra12}
\end{figure}

Let $S$ be any subset of the vertices of $V'$ of cardinality strictly smaller than $3n/2$. Since this means $S$ contains less than half of the vertices of $V'$, this means that the complement of $S$ contains a pair of adjacent vertices by the pigeonhole principle. Since one of the $\Delta^{w_i}_{W,V'}$ contains all vertices of $V'$ except for this pair, we conclude that the convex hull of $S$ is contained in this $\Delta^{w_i}_{W,V'}$. This means that any subset of less than $3n/2$ columns of $M'$ has nonnegative rank at most $3$, while $M'$ itself has nonnegative rank at least $4$. Note that this proof is analogous to that of Moitra, barring the fact that we can take any subset of cardinality strictly smaller than $3n/2$, rather than any subset of cardinality strictly smaller than $n$. \qed
\end{ex}

We have seen that there is no stabilization property on the topological boundary of matrices with given nonnegative rank. The reader might wonder if this is true more generally for the algebraic boundary as well? Despite Moitra's example (for the topological boundary), for $r=3$ the stabilization on the algebraic boundary is true. A matrix $M \in \mathbb{R}^{m \times n}$ not containing zeros lies on the algebraic boundary $\overline{\partial (\mathcal{M}_{m\times n}^3)}$ if and only if it has a size three factorization $AB$ with seven zeros in special positions. If $n >4$, then we can find a column $i_0$ of $B$ that does not contain any of these seven zeros. Let $M^{\widehat{i_0}}$ and $B^{\widehat{i_0}}$ be obtained from $M$ and $B$ by removing the $i_0$-th column. Then $M^{\widehat{i_0}}$ has the factorization $AB^{\widehat{i_0}}$ with seven zeros in special positions, and hence lies on $\overline{\partial (\mathcal{M}_{m\times n-1}^3)}$.
For grater $r$ we formulate the following conjecture for columns (it could be formulated for rows as well).

\begin{conjecture}\label{conjecture1}
For given $r \geq 3$ there exist $n_0 \in \mathbb{N}$, such that for all $n\geq n_0$ and for all matrices $M=(m_i)_{i=1,...,n}$ on the algebraic boundary $\overline{\partial (\mathcal{M}_{m\times n}^r)}$ there is  a column $1\leq i_0\leq n$ such that the truncated matrix $M^{\widehat{i_0}}=(m_i)_{i=1,...,n}^{i\neq i_0}$ lies on the algebraic boundary $\overline{\partial(\mathcal{M}_{m\times n-1}^r)}$.
\end{conjecture}
In the construction of Moitra's example it was crucial that both the number of rows and the number of columns was let to tend to infinity. One might hope that the topological boundary stabilizes if the number of rows (or columns) is kept fixed. Unfortunately, not even in this restricted case, the stabilization of the topological boundary is true. Robert Krone has a family of matrices ($m=5$ and arbitrary $n$) of nonnegative rank at least $4$  such that removing any column of the matrix gives a matrix of nonnegative rank $3$~\cite{RK}.

It is not clear though that such a family of examples is constructible for arbitrary $m$. This question seems to be related to the question regarding the existence of so called "maximal configurations" in ~\cite[Section 5]{MSS}. For the $m=4$ case the maximal boundary configuration we managed to construct has $8$ points, so $n=8$. That is the following example.

\begin{ex} Let $W$ be a square, and orient its edges counterclockwise. For every vertex $w$ of $W$ and for every angle $\theta$, let $l_{w,\theta}$ be the line that is at an angle $\theta$ to the unique directed edge starting at $w$. For fixed $\theta$ with $0 \leq \theta \leq \pi/4$, let $V_{\theta}$ be the polygon cut out by the lines $l_{w,\theta},l_{w,\pi/2-\theta}$ with $w$ running over the vertices of $W$.

By construction, for any edge $e$ of $V$, any of the triangles $\Delta^e$ is one of the $\Delta^{w}$. The left side of Figure~\ref{fig:basicexample} shows the square $W$, the octagon $V_{\pi/8}$, and the triangles $\Delta^w,\Delta^e$ (some of which coincide for this $\theta$, but not in general).

Note that $V_{\theta}$ can have at most $8$ vertices, so any pair $(W,V_{\theta})$ can be obtained from some matrix in $M_{4,8}^3$.
Observe that $V_{\theta'}$ lies in the interior of $V_{\theta}$ for all $0 \leq \theta < \theta' \leq \pi/4$, meaning that there can be at most one $\theta$ for which the pair $(W,V_{\theta})$ corresponds to some $M \in \partial\mathcal{M}_{4,8}^3$. Moreover, such $\theta$ exists. This follows from the fact that for $\theta = 0$, we have $V_{\theta} = W$ and does not have nonnegative rank $3$, and for $\theta = \pi/4$, the space $V_{\theta}$ consists of a single point, and hence has nonnegative rank at most $3$.

By direct computation, one can show that $V_{\pi/8} \in \partial\mathcal{M}_{4,8}^3$. From here on, we simply write $V = V_{\pi/8}$. Again, have a look at the left part of Figure~\ref{fig:basicexample}. You can see from the picture that all bounding triangles are tight, and in particular, $V$ does not lie in the interior of any triangle between $V$ and $W$.

The vertices of $V$ are of two types, namely those lying on the angle bisectors of the vertices of $W$ and those lying on the perpendicular bisectors of the edges of $W$. We call vertices of the first type \emph{angular vertices} and vertices of the second type \emph{perpendicular vertices}.

We modify $V$ to $V'$ by moving the perpendicular vertices a distance $\epsilon$ outwards along the bisectors. We observe that any triangle containing $V'$ must also contain $V$. Since $V$ is only contained in the triangles $\Delta^w$ with $w$ running over the vertices of $W$, and since $\Delta^{w}$ does not contain the angular vertex across from it (as can be seen by looking at the red triangle on the right hand side of Figure~\ref{fig:basicexample}), this means that $V'$ is not contained in any triangle that is contained in $W$.

Suppose $\epsilon$ is sufficiently small. If one removes an angular vertex $v$, we see that all remaining vertices of $V'$ are contained in one of the triangle $\Delta^w_{V',W}$ where $w$ is the vertex of $W$ across from $v$, as is demonstrated by the red triangle on the right hand side of Figure~\ref{fig:basicexample}. In terms of matrices, if one removes any column corresponding to an angular vertex, the resulting matrix will have nonnegative rank $3$.

If one removes a perpendicular vertex $v$, things are slightly more tricky. The new polygon $V''$ will contain an extra edge $e$. By direct calculation, we can show that $\Delta^e_{V'',W}$ is contained in $W$ (and in fact, this is a tight fit). This can be seen by looking at the blue triangle on the right hand side of Figure~\ref{fig:basicexample}. So again, if one removes a column corresponding to a perpendicular vertex, the resulting matrix will have nonnegative rank $3$.

We conclude that the pair $(W,V')$ has nonnegative rank $4$, and that if one removes any column from the corresponding matrix, the result has nonnegative rank $3$.
\qed
\end{ex}

\begin{figure}[h]
\includegraphics[width=12cm]{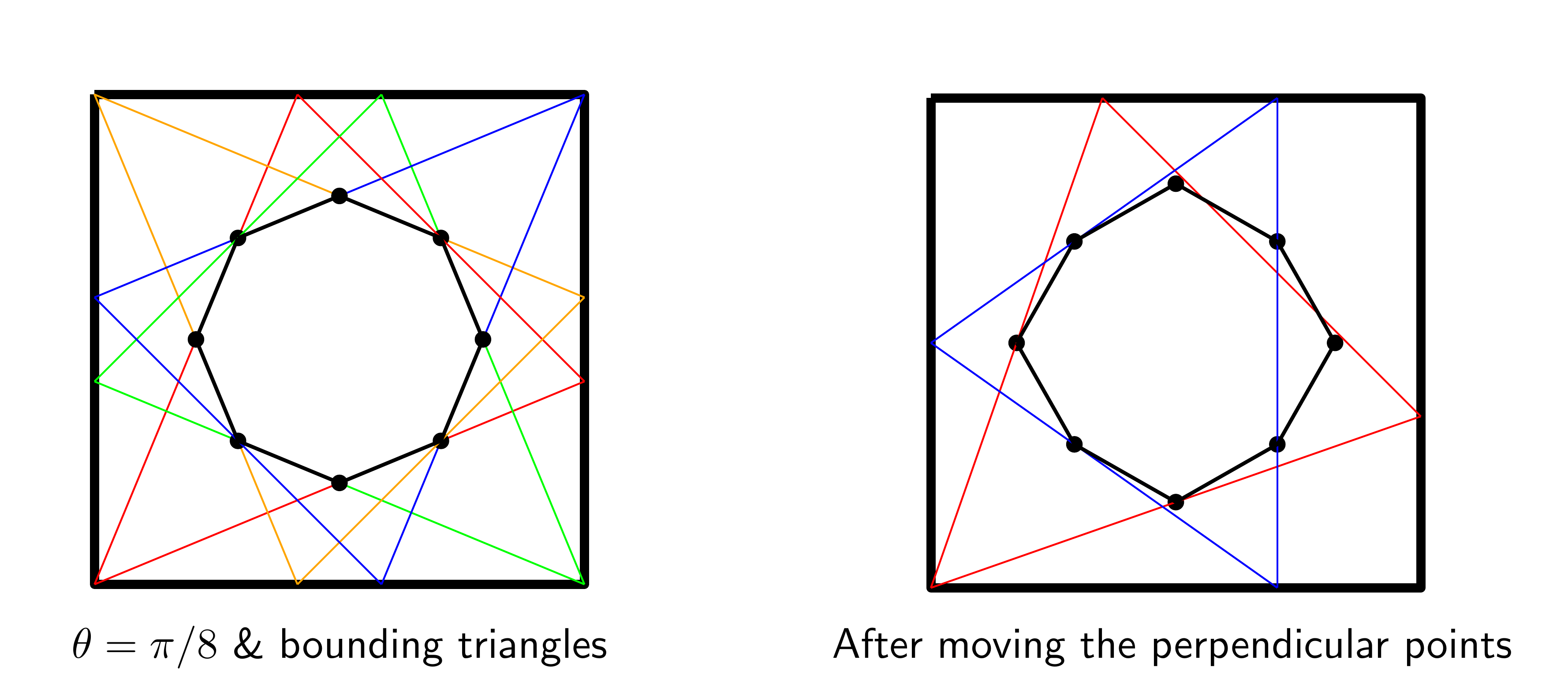}
\caption{The example, before and after moving points}\label{fig:basicexample}
\end{figure}

In the above example, we do not know the reason why $\Delta^e_{V'',W}$ is contained in $W$ (and why it is a tight fit). Numerical approximations suggest that a similar statement is true when $m = 8$ (and $n=16$), so some more general statement might be true.

If a similar statement is true when we replace $W$ by a regular $m$-gon (with $m$ not divisible by $3$), then we can generalize this example to a family of $m\times 2m$ examples similar to Moitra's family of examples, but with the property that the non-negative rank drops whenever one removes a single vertex (rather than whenever one removes a subset of the vertices of high cardinality). We can generalize the example even if such a property does not hold, but it would force us to modify $W$ to $W'$ as well as modifying $V$ to some $V'$.

We have seen that certain properties of the space of factorizations are influencing whether a configuration lies on the boundary.
A slightly milder approach to the stabilization property, would be to examine the local behavior of the space of factorizations.
A matrix on the boundary of the mixture model has only very restricted nonnegative factorizations (even only finitely many for $r=3$, see \cite[Lemma 3.7]{MSS}) and it might be true that stabilization holds locally for each particular factorization of the model. Of course by deleting a column (or a corresponding point) new factorizations may appear, so we can not say anything globally. We formulate this idea in the following conjecture for columns (it could be formulated for rows as well).

\begin{conjecture}\label{conjecture2}
For given $r\geq 3$ there exists an $ n_0 \in \mathbb{N}$, such that for all $n\geq n_0$ and for all nonnegative factorizations $M=AB$ where $M$ is on the topological boundary $\partial (\mathcal{M}_{m\times n}^r)$ there is a column $1\leq i_0\leq n$ and an $\epsilon >0$ such that in the $\epsilon$-neighborhood of the nonnegative factorization $AB^{i_0}$ all size $r$ factorizations of $M^{i_0}$ are obtained from factorizations of $M$ by removing the $i_0$-th column.
\end{conjecture}

In the nonnegative rank $3$ case a matrix lies on the topological boundary if and only if all nonnegative factorizations have seven zeros in special positions (which are isolated points in the space of factorizations, see \cite[Lemma 3.7]{MSS}), whereas it lies on the algebraic boundary if and only if it has at least one factorization with seven zeros in special positions (there exists an isolated factorization). So in the nonnegative rank $3$ case the above conjecture is true and it is equivalent to Conjecture~\ref{conjecture1}. For higher $r$ the two conjectures are not equivalent, but Conjecture~\ref{conjecture1} implies Conjecture~\ref{conjecture2}.

For arbitrary $r$ we can prove this conjecture for a special case. Assume that $M$ lies on the topological boundary and it has a factorization such that not all vertices of the interior polytope $V$ lie on the boundary of $\Delta$. Let $v$ be one such vertex. We can remove the column corresponding to $v$ and choose $\epsilon$ less than the distance of $v$ to the closest facet of $\Delta$. Then $v$ does not lie on the boundary of $\Delta$ for any simplex $\Delta$ in an $\epsilon$ neighborhood of $\Delta$. In particular, $v$ does not influence whether $\Delta$ contains the interior polytope $V$ in this neighborhood, hence we can remove this vertex.

\appendix
\section{Gr\"obner Basis Computations}\label{section:appendix}

The \texttt{Macaulay2} code for the equivariant Gr\"obner basis computation is:
\begin{lstlisting}
m=8;
n=10;
R1=QQ[p_(1,1)..p_(4,6)];
--ideal for m=4 and n=6
S=QQ[a_(1,1)..a_(4,3),b_(1,1)..b_(3,6)];
M1=matrix{{0,a_(1,2),a_(1,3)},{0,a_(2,2),a_(2,3)},
   {a_(3,1),0,a_(3,3)},{a_(4,1),a_(4,2),0}};
M2=matrix{{0,b_(1,2),b_(1,3),b_(1,4),b_(1,5),b_(1,6)},
   {b_(2,1),0,b_(2,3),b_(2,4),b_(2,5),b_(2,6)},
   {b_(3,1),b_(3,2),0,b_(3,4),b_(3,5),b_(3,6)}};
M=M1*M2;
f=map(S,R1,flatten flatten entries M);
I1=kernel f;

--separate the degree 6 generators in the ideal

I1deg6={};
for i to numgens(I1)-1 do (
    if((degree(I1_i))#0==6) then I1deg6=append(I1deg6,I1_i);
)

--construct a new ideal that is the orbit of the
--original ideal under G

R2=QQ[p_(1,1)..p_(m,n)];

--construct G that fixes {1,2,3} and
--maps increasingly {4,5,6} into {4,...,n}

l=for i from 4 to n list i;
IncTemp=subsets(l,3);
Inc=for i to #IncTemp-1 list(join({1,2,3},IncTemp#i));

--construct substitute list for each element of G

substituteList=for i to #Inc-1 list
for j to numgens(R1)-1 list
    R1_j=>p_((baseName R1_j)#1#0,Inc#i#((baseName R1_j)#1#1-1));
substituteList

--construct all 4x4 minors of the mxn matrix

P=matrix(pack(n,flatten entries vars R2));
I2deg4=minors(4,P);

--construct degree 6 polynomials

I2deg6list=flatten for i to #I1deg6-1 list
for j to #substituteList-1 list
    sub(I1deg6#i,substituteList#j);
I2deg6list=unique(I2deg6list);
I2=I2deg4+ideal(I2deg6list);
sort gens gb I2 == sort gens I2
\end{lstlisting}

\end{document}